\documentclass [10pt]{article}
\usepackage{graphicx,subfigure}
\usepackage{amsmath}
\usepackage{amsfonts}
\usepackage{amssymb}
\usepackage{url}
\usepackage{float,color,comment,tabulary}
\usepackage{hyperref}
\usepackage{xcolor}

\newfloat{figure}{H}{lof}
\newfloat{table}{H}{lot}
\floatname{figure}{\figurename}
\floatname{table}{\tablename}
\newtheorem{theorem}{Theorem}[section]

\newtheorem{lemma}[theorem]{Lemma}

\newtheorem{remark}[theorem]{Remark}

\numberwithin{equation}{section}
\newenvironment{proof}[1][Proof]{\noindent\textit{#1.} }{\hfill \rule{0.5em}{0.5em}}

\newcommand{\R}{\mathbb{R}}
\newcommand{\A}{\mathbb{A}}
\newcommand{\LL}{\mathbb{L}}
\newcommand{\X}{\mathbb{X}}
\newcommand{\F}{\mathbb{F}}
\newcommand{\BUC}{{\rm BUC}}

\renewcommand{\eqref}[1]{(\ref{#1})}

\begin{document}

\title{\textbf{Return-to-home model for short-range human travel}}
	\author{\textsc{Arnaud Ducrot$^{(a)}$, and Pierre Magal$^{(b),}$\thanks{Corresponding author. e-mail: \href{mailto:pierre.magal@u-bordeaux.fr}{pierre.magal@u-bordeaux.fr}}}\\
	{\small \textit{$^{(a)}$ Normandie Univ, UNIHAVRE, LMAH, FR-CNRS-3335, ISCN, 76600 Le Havre, France}} \\
	{\small \textit{$^{(b)}$ Univ. Bordeaux, IMB, UMR 5251, F-33400 Talence, France}} \\
	{\small \textit{CNRS, IMB, UMR 5251, F-33400 Talence, France.}}
	}
\maketitle

\begin{center}
	\textit{This article is dedicated to the memory of Professor Stephen Gourley, a very fine gentleman.}
\end{center}
\begin{abstract}
In this work, we develop a mathematical model to describe the local movement of individuals by taking into account their return to home after a period of travel. We provide a suitable functional framework to handle this system and study the large-time behavior of the solutions. We extend our model by incorporating a colonization process and applying the return to home process to an epidemic.
\end{abstract}

	\noindent \textbf{Keywords:} Return to home process,  Diffusion process; Colonization process; Epidemic model

\section{Introduction}

Understanding human displacement  is very important since it influences the population's dynamic. In light of the recent COVID-19 epidemic outbreak, human travel is critical to understand how a virus spreads at the scale of a city, a country, and the scale of the hearth, see \cite{Ruan, Ruan2017}. Human mobility is also essential to understand and quantify the changes in social behavior.   

The spatial motion of populations is sometimes modeled using Brownian motion and diffusion equations at the population scale. For instance, reaction-diffusion equations are widely used to model the spatial invasion of populations both in ecology and epidemiology. We refer for instance Cantrell and Cosner \cite{Cantrell}, Cantrell, Cosner and Ruan \cite{Cantrell-Cosner-Ruan}, Murray \cite{Murray}, Perthame \cite{Perthame}, Roques \cite{Roques} and the references therein. In particular, the spatial propagation for the solutions of reaction-diffusion equations has been observed and studied in the 30s by Fisher \cite{F}, and Kolmogorov, Petrovski, and Piskunov \cite{KPP}. Diffusion is a good representation of the process of invasion or colonization for humans and animals. Nevertheless, once the population is established, the return to home process (i.e., diffusion combined with a return to home) seems to be more adapted to the daily life of humans. 

This article aims at describing short-distance human mobility. We aim at considering the local human movement at the scale of the city with a special focus on how to model the returning home of humans. The return to home behavior is also very important in ecology, agriculture, fisheries. In agriculture, a return to home model for bees has been proposed and studied in \cite{MWW1}, and \cite{MWW2}. 

Roughly speaking, we can classify the human movement into 1) short-distance movement: working, shopping, and other activities; 2) long-distance movement: intercity travels, plane, train, cars, etc. These considerations have been developed recently in \cite{BHG} \cite{GHB}  \cite{KSZ} \cite{MS}. A global description of the human movement has been proposed (by extending the idea of the Brownian motion) by considering the Lévy flight process. The long-distance movement can be covered by using patch models (see Cosner et al.	\cite{CBCI} for more results about this).

In section 2, we present the return to home model. Section 3 explores some conversations properties of the return to home model. Section 4 presents a semi-explicit formula for the solutions of the return to home model. Section 5 studies the equilibrium distributions. Section 6 develops a functional framework to understand the mild solutions for the return to home model. Section 7 is devoted to an extension  Fisher KPP model with return to home and colonization. By analogy to the Fisher KPP problem, we obtain a monotone semiflow. The last section is devoted to an epidemic model with return to home.

%The dynamic spatial redistribution of individuals is a key driving force of various spatiotemporal phenomena on geographical scales. 	
%
%Human travel, for example, is responsible for the geographical spread of human infectious disease4–9. 
%
%
%The travel of human have studied by many people Ref????. Recently some long and short-range travel have been observed by \cite{BHG} and end up to be a scaling laws of human travel
%
%\textit{As explained in  \cite{BHG}, the distribution of travelling distances decays as a power law, indicating that trajectories of bank notes are reminiscent of scale- free random walks known as L\'evy flights. }
%
%
%\medskip 
%??? home "visitation model"???
%
%
%- récurrence ou transiance du mouvement brownien 
%
%
%\newpage 
\section{Return to Home Model}
In this section we describe a model for the movement of individuals within a city.  To simplify the presentation, we consider a population moving in the whole plane $\R^2$.  Our goal is to focus on the most important processes involved in the model. The presentation would be more complicated for a model in a bounded set taking into account boundary conditions.  

\medskip 
\noindent \textbf{Model for individuals staying at home:}  Let $y \to u(t,y)$ be the distribution of population of individual staying at home at time $t$.  This means that the quantity
$$
\int_{\omega}u(t,y)dy
$$ 
represents the number of individuals staying at home in the sub-region $\omega \subset \R^2$ at time $t$. 

The following equation describes the flux between individuals staying at home at the location $y \in \R^2$ and individuals out of the house 
\begin{equation} \label{2.1}
\partial_t u(t,y)= \underset{\text{Flow of individuals returning home}}{\underbrace{ \alpha \int_{\R^2}v(t,x,y)dx}}- \underset{\text{Flow of individuals leaving home}}{\underbrace{\gamma u(t,y) }}
\end{equation}
where $1/\alpha>0$ is the average time spent by individuals at home, and $1/\beta>0$ is the average time spent by individuals out of their house.

\medskip 
\noindent \textbf{Model for travelers (people who are not staying at home):} Let $x \to v_y(t,x)=v(t,x,y)$ be the distribution of population of individuals  out of the house (called travelers) and originated from their home located at $y \in \R^2$. Here originated from $y \in \R^2$ means that their home is located at the position $y$.  Then the quantity  
$$
\int_{\omega}v(t,x,y)dx
$$ 
is the number of individuals (originated from the home located at $y \in \R^2$) and traveling in the sub-region $\omega \subset \R^2$  at time $t$.  

The following equation describes the flux between individuals staying at home at the location $y$ and the travelers 
\begin{equation} \label{2.2}
	\begin{array}{ll}
	\partial_t v(t,x,y)&= \underset{\text{ Brownian motion}}{\underbrace{  \varepsilon^2 \bigtriangleup_x v(t,x,y)}} -  \underset{\text{Flow of individuals returning home}}{\underbrace{ \alpha \, v(t,x,y)}}\\ \\&+ \underset{\text{Flow of individuals leaving home}}{\underbrace{\gamma \, \rho(x-y) \, u(t,y)}},
	\end{array}
\end{equation}
where $\bigtriangleup_x$ denotes the Laplace operator for the the variable $x=(x_1,x_2)\in\R^2$, that is
$$
\bigtriangleup_x=\partial_{x_1}^2+\partial_{x_2}^2.
$$ 
In equation \eqref{2.2}, $\varepsilon>0$ is the diffusivity of the travelers.  The map $x \to \rho(x-y)$ is a Gaussian distribution representing the location of a house centered at the position $y \in \R^2$. Here by a Gaussian distribution centered at $0$ we mean that the function $\rho$ is given by  
\begin{equation} \label{2.3}
\rho(x_1,x_2)= \dfrac{1}{2 \pi  \sigma^2 } e^{- \dfrac{x_1^2+x_2^2}{2 \sigma^2 } },
\end{equation}
for some $\sigma^2>0$, so that the covariance matrix is given by the diagonal matrix $\sigma^2 I_2$. 
Note also that for all $y\in \R^2$, the translated map $\rho(\cdot-y)$ satisfies 
$$
 \int_{\R^2}   \rho(x-y)dx=1 \text{ and } \int_{\R^2} x \rho(x-y)dx=y.  
$$

\medskip 
\noindent \textbf{Initial condition:} System  \eqref{2.1} and \eqref{2.2} is complemented with some initial distributions 
\begin{equation} \label{2.4}
u(0,y)=u_0(y) \text{ and } v(0,x,y)=v_0(x,y). 
\end{equation}

\medskip

%\begin{remark}
%We can also consider more general $2$-dimensional Gaussian distribution such as  
%	$$
%	\rho(x_1,x_2)= \dfrac{1}{\sigma_{x_1} \sqrt{2 \pi }} e^{- \dfrac{x_1^2}{2 \sigma_{x_1}^2 } } \times  \dfrac{1}{\sigma_{x_2} \sqrt{2 \pi }} e^{- \dfrac{x_2^2}{2 \sigma_{x_2}^2 } }. 
%	$$
%\end{remark}
\begin{remark}
Here in the model we have neglected possible convection term describing for instance the transport of individuals from their home to their workplace, shopping mall etc. Such a convection term describes the tendency of individuals to start moving from their home location $y$: 
	\begin{equation*}  
		\begin{array}{ll}
			\partial_t v(t,x,y)&=\varepsilon^2 \bigtriangleup_x v(t,x,y)+ \underset{\text{convection}}{\underbrace{\left( c_y(x) \bigtriangledown v(t,x,y) 	\right)}}
		\\
			\\
			& -  \alpha \, v(t,x,y)+\gamma \, \rho(x-y) \, u(t,y),
		\end{array}	
	\end{equation*}
where $ c_y (x) $ is the travel speed at location $ x $. This speed may depend on the location of the house $ y $ to distinguish the individual's origin in the city.
	\end{remark}

\begin{remark} If we formally replace $u(t,y)$ by $ \dfrac{\alpha}{\gamma}\int_{\R^2}v(t,x,y)dx$ in the $v$-equation (see \eqref{2.2}) of the model, we obtain the following single non local equation  
	
	\begin{equation*}  
	\begin{array}{ll}
		\partial_t v(t,x,y)&=    \varepsilon^2 \bigtriangleup_x v(t,x,y) +  \alpha \, \rho(x-y) \,   \int_{\R^2}v(t,x,y)dx -   \alpha \, v(t,x,y).
	\end{array}
	\end{equation*}

\end{remark}

\section{Remarkable distributions}
\subsection{Total distribution of population in space} 
The total distribution of population originated from a house located at $y \in \R^2$ is 
\begin{equation*}  
w(t,x,y)=v(t,x,y)+u(t,y)  \rho(x-y).
\end{equation*}
By integrating the distribution with respect to $y$ (homes locations), we obtain the total distribution of population in space (including both individuals staying at home or travelers)   
\begin{equation*}  
\overline{w}(t,x)=\int_{\R^2} w(t,x,y) dy . 
\end{equation*}

\subsection{The total distribution of individuals at home}  
The total distribution of individuals at home reads as  
\begin{equation*}
h(t,y)= u(t,y)+\int_{\mathbb R^2}v(t,x,y)dx.  
\end{equation*}
Data representing this distribution is usually available. This is true for example in USA \cite{Data} where such a distribution is known per county (a home is identified to a county).

If we neglect the newborns and the deaths, and people moving house, we can assume that  $h(t,y)=h_0(y)$ is independent of time. Let $y \to h_0(y)$ be  the spatial distribution of individuals with their home located at the position $y$. This means that
$$
\int_{\omega}h_0(y)dy 
$$  
denotes the number of people having their home in the region $\omega \subset \R^2$. 

In that case, we obtain 
\begin{equation} \label{3.1}
h_0(y)=u(t,y)+\int_{\mathbb R^2}v(t,x,y)dx, 
\end{equation}
or equivalently 
\begin{equation} \label{3.2}
\int_{\mathbb R^2}v(t,x,y)dx=h_0(y)-u(t,y),\,\forall t\geq 0. 
\end{equation}
\section{Explicit formula for the solutions}  
By using equations \eqref{2.1} and \eqref{2.2}, we have 
$$
\frac{\partial}{\partial t}h(t,y)=\frac{\partial}{\partial t}\left[u(t,y)+\int_{\mathbb R^2}v(t,x,y)dx\right]=0,
$$
Therefore  the solution \eqref{2.1} and \eqref{2.2} satisfies \eqref{3.1}. 

\medskip 
By replacing the formula \eqref{3.2} in the $u$-equation \eqref{2.1}, we obtain 
\begin{equation*}  
\partial_t u(t,y)=\alpha h_0(y)-(\alpha+\gamma)u(t,y),
\end{equation*}
and we deduce an explicit form for the solutions  
\begin{equation} \label{4.1}
u(t,y)=u_0(y)e^{-(\alpha+\gamma)t}+\frac{\alpha}{\alpha+\gamma}(1-e^{-(\alpha+\gamma)t})h_0(y).
\end{equation}
In order to derive an explicit formula for $v(t,x,y)$ we define the two-dimensional heat kernel $K:(0,\infty)\times\R^2\to \R$ be
\begin{equation*}   
K(t,x)=  \dfrac{1}{4\pi\varepsilon^2 t} e^{-\frac{|x|^2}{4\varepsilon^2t}},
\end{equation*}
and we make the simplifying assumption. 
\begin{lemma} \label{LE4.1}Assume that $\rho(x)$ is a Gaussian distribution \eqref{2.3}. Then there exists $t_0>0$ such that 
	$$
	\rho(x)=K(t_0,x).
	$$
Moreover the convolution of $\rho$ with the kernel $K(t,.)$ satisfies 
	\begin{equation} \label{4.2}
		\int_{\mathbb R^2}K(t,x-z) \rho(z-y)d z =K(t+t_0,x-y).
	\end{equation}
\end{lemma} 
\begin{proof} We assume that fact that 
	$$
	\rho(x)= K(t_0,x),
	$$
	we deduce that 
	$$
	\int_{\mathbb R^2}K(t,x-z)   {\rho}(z-y)d z = \int_{\mathbb R^2}K(t,x-y-z)   \widehat{\rho}(z)d z =\int_{\mathbb R^2}K(t,x-y-z)  K(t_0,z)d z 
	$$
	and the result follows by the semigroup property of the diffusion semigroup. 
\end{proof}

\bigskip 
\noindent Define 
\begin{equation*}  
T_{ \varepsilon^2 \bigtriangleup_x-\alpha I }(t)\left( v(.,y) \right) (x)=e^{-\alpha t} \int_{\mathbb R^2} K(t,x-z)  v(z,y)d z,
\end{equation*}
and 
\begin{equation*}  
	T_{ \varepsilon^2 \bigtriangleup_x  }(t)\left( v(.,y) \right) (x)=  \int_{\mathbb R^2} K(t,x-z)  v(z,y)d z,
\end{equation*}
where 
\begin{equation*}   
K(t,x)=  \dfrac{1}{4\pi\varepsilon^2 t} e^{-\frac{|x|^2}{4\varepsilon^2t}}.
\end{equation*}
We obtain 
\begin{equation*}  
	\begin{array}{ll}
	v(t, x,y)& =T_{ \varepsilon^2 \bigtriangleup_x-\alpha I }(t)\left( v_0(.,y) \right) (x)\\
		
	&+ \int_0^t T_{ \varepsilon^2 \bigtriangleup_x-\alpha I }(t-\sigma )\left(  \gamma \rho(.-y) u(\sigma, y) \right) (x) d \sigma.
	\end{array}
\end{equation*}
Next, by  replacing the explicit formula \eqref{4.1} for $u(t,y)$ in the above formula, we obtain 
\begin{equation}  \label{4.3}
v(t,x,y)= \left( v_1+ v_2+v_3 \right)(t,x,y),
\end{equation}
with
\begin{equation}  \label{4.4}
v_1(t,x,y)= e^{-\alpha t} \int_{\R^2} K(t,x-z) v_0(z,y)dz, 
\end{equation}
$$
v_2(t,x,y)= \int_0^t e^{-\alpha (t-\sigma) }v_{ \varepsilon^2 \bigtriangleup_x  }(t-\sigma )\left(  \gamma \rho(.-y) \left[u_0(y)-\frac{\alpha}{\alpha+\gamma} h_0(y) \right]e^{-(\alpha+\gamma) \sigma } \right) (x) d \sigma,
$$
and
$$
v_3(t,x,y)= \int_0^t e^{-\alpha (t-\sigma) }T_{ \varepsilon^2 \bigtriangleup_x  }(t-\sigma )\left(  \gamma \rho(.-y) \frac{\alpha}{\alpha+\gamma}h_0(y) \right) (x) d \sigma. 
$$
By using the above formula for $v_2$, we obtain  
\begin{equation*}
v_2(t,x,y)= \gamma  e^{-\alpha t} \int_0^t e^{-\gamma \sigma }  T_{ \varepsilon^2 \bigtriangleup_x  }(t-\sigma )\left(  \rho(.-y) \right) (x) d \sigma \left[u_0(y)-\frac{\alpha}{\alpha+\gamma} h_0(y) \right],
\end{equation*}
and by using Lemma \ref{LE4.1} and formula \eqref{4.2}, we deduce that 
\begin{equation}  \label{4.5}
v_2(t,x,y)=   e^{-\alpha t} \int_0^t e^{-\gamma \sigma } K(t+t_0-\sigma,x-y) d \sigma \, \gamma  \left[u_0(y)-\frac{\alpha}{\alpha+\gamma} h_0(y) \right].
\end{equation}
By using the expression for $v_3$, we obtain  
 $$
 v_3(t,x,y)= \int_0^t e^{-\alpha (t-\sigma) }K(t+t_0-\sigma,x-y) d \sigma \frac{\gamma \alpha}{\alpha+\gamma}h_0(y). 
 $$
 By using the change of variable $\widehat{\sigma}=t-\sigma$ we obtain 
 \begin{equation} \label{4.6}
 	 v_3(t,x,y)= \int_0^t \alpha e^{-\alpha \widehat{\sigma} }K(\widehat{\sigma}+t_0,x-y) d \widehat{\sigma} \frac{\gamma }{\alpha+\gamma}h_0(y) .
 \end{equation}
Therefore, combining equations \eqref{4.3}-\eqref{4.6}, we obtain an explicit formula for $v(t,x,y)$.

\section{Equilibrium and their stability}
The equilibrium distribution must satisfy the following coupled system of equations 
\begin{equation} \label{5.1}
	0= \alpha \int_{\R^2} \overline{v}(x,y)dx- \gamma \overline{u}(y), 
\end{equation}
and
\begin{equation} \label{5.2}
	0= \varepsilon^2 \bigtriangleup_x \overline{v}(x,y)-  \alpha \, \overline{v}(x,y)+\gamma \, \rho(x-y) \,  \overline{u}(y). 
\end{equation}
The first equation \eqref{5.1} provides 
\begin{equation} \label{5.3}
	\overline{u}(y)=\dfrac{\alpha}{\gamma} \int_{\R^2} \overline{v}(x,y)dx.
\end{equation}
By using in \eqref{5.2} we obtain 
\begin{equation*}
	\overline{v}(x,y)=\left( \alpha I-  \varepsilon^2 \bigtriangleup_x \right)^{-1} \left(  \rho(.-y) \,  \gamma \,  \overline{u}(y) \right) 
\end{equation*}
and since $ \overline{H}(y) $ is independent of $x$, we obtain 
\begin{equation*}
	\overline{v}(x,y)=\left( \alpha I-  \varepsilon^2 \bigtriangleup_x \right)^{-1} \left(    \rho(.-y) \,  \right)  \,  \gamma \, \overline{u}(y). 
\end{equation*}
and 
\begin{equation*}
	\begin{array}{ll}
		\left( \alpha I-  \varepsilon^2 \bigtriangleup_x \right)^{-1} \left( \rho(.-y) \,  \right)(x) & =	 \int_0^{\infty}     e^{-\alpha t} \int_{\R^2}  K(t,x-z)   \rho(z-y) dz dt \vspace{0.2cm}\\
		&= \int_0^{\infty}     e^{-\alpha t} K(t+t_0,x-y)    dt . 
	\end{array}
\end{equation*}
Define  now
$$
\chi(z)= \gamma \, \int_0^{\infty}     e^{-\alpha t} K(t+t_0,z)    dt,
$$
so that we get 
\begin{equation}
	\overline{v}(x,y)= \chi(x-y) \overline{u}(y). 
\end{equation}
\begin{remark}
	We observe that
	$$
 \int_{\R^2}	\chi(x-y)dx =\dfrac{\gamma}{\alpha}. 
	$$
	Therefore by using \eqref{3.1} we obtain 
	$$
	h_0(y)= \left[1+\dfrac{\gamma}{\alpha} \right] \overline{u}(y) \Leftrightarrow \overline{u}(y) =\dfrac{\alpha}{\alpha+ \gamma } 	h_0(y). 
	$$

\end{remark}
\section{Mild solutions for the non-homogeneous equation}

\subsection{A change of variable of some interest}
Recall that the return to home model reads as follows 
\begin{equation} \label{6.1A}
	\left\{
	\begin{array}{lll}
		\partial_t u(t,y)=& \, \, \, \, \alpha \int_{\R^2}v(t,x,y)dx &-\gamma u(t,y), \vspace{0.2cm}\\
		\partial_t v(t,x,y)=  \varepsilon^2 \bigtriangleup_x v(t,x,y)&-\alpha \,  v(t,x,y)&+\gamma \, \rho(x-y) \, u(t,y),
	\end{array}
	\right. 
\end{equation}
In the above formulation of the return to home model $x$ denotes the position of travelers and their home located at $y$. So we use an Eulerian system of coordinate for $x$ is independent of the home location $y$. Instead, we can consider the spatial location 
$$
z=x-y
$$ 
which can be regarded as a Lagrangian system of coordinate for $z$ centered at the home location $y$.

By using this change of variables
$$
w(t,z,y)=v(t,z+y,y) \Leftrightarrow w(t,x-y,y)=v(t,x,y) 
$$ 
we obtain the following system of equations
\begin{equation} \label{6.1B}
	\left\{
	\begin{array}{l}
		\partial_t u(t,y)=-\gamma u(t,y)+ \alpha \int_{\R^2}w(t,z,y)dz, \vspace{0.2cm}\\
		\partial_t w(t,z,y)=  \varepsilon^2 \bigtriangleup_z w(t,z,y)-\alpha \, w(t,z,y)+\gamma \, \rho(z) \, u(t,y).
	\end{array}
	\right. 
\end{equation}

In order understand our choice of Banach spaces, one can observe that the term in \eqref{6.1A}
$$
\gamma \, \rho(x-y) \, u(t,y)
$$ 
which becomes in \eqref{6.1B}
$$
\gamma \, \rho(z) \, u(t,y)
$$ 
This explain our choice of the Banach $Y_2$ is the following. 
\subsection{Some Banach spaces}
We consider $\BUC\left(\R^2\right)$ the space of bounded and uniformly continuous maps  from $\R^2$ to $\R$, which is a Banach space endowed with the supremum norm 
$$
\Vert u \Vert_{\infty}= \sup_{y \in \R^2} \vert u(y)\vert. 
$$
\noindent \textbf{The space $Y_1$:} We define 
$$
Y_1=\BUC \left(\R^2, L^1 \left(\R^2\right) \right),
$$
the space of maps $x \to v(x,.)$ which belongs to $\BUC  \left(\R^2, L^1 \left(\R^2\right) \right)$.  The space $ \BUC  \left(\R^2, L^1 \left(\R^2\right) \right)$ becomes a Banach space when it is endowed with the norm 
$$
\Vert v \Vert_{Y_1}=\sup_{x \in \R} \Vert v(x,.) \Vert _{L^1 \left(\R^2\right) }=\sup_{x \in \R}  \int_{\R^2}  \vert v(x,y)\vert dy . 
$$
\noindent \textbf{The space $Y_2$:} We also define $Y_2$ the space of maps $(x,y)\to  v(x,y)$ such that the function $
w(x,y)=v(x+y,y)$ satisfies 
$$
x\to  w(x,.) \in L^1 \left(\R^2, \BUC \left(\R^2\right) \right).
$$
That is also equivalent to say that $Y_2$ is the space of maps $(x,y)\to  w(x-y,y)$ such that 
$$
x \to w(x,.) \in L^1 \left(\R^2, \BUC \left(\R^2\right) \right). 
$$
Therefore $ Y_2$ is a Banach space endowed with the norm 
$$
\Vert v \Vert_{Y_2}  = \int_{\R^2} \sup_{y \in \R^2}  \vert v(x+y,y)\vert dx . 
$$
First note that the maps in $Y_2$ enjoy the following property.
\begin{lemma} \label{LE6.1} Let $v \in Y_2$ be given. Then the map  
	$$
	y \to \int_{\R^2}    v(x,y) dx .
	$$
	is bounded and uniformly continuous on $\R^2$. 
	\end{lemma} 
\begin{proof}
	By construction the map 
	$$ 
		y \to \int_{\R^2}  v(x+y,y) dx 
	$$
is bounded and uniformly continuous. But by using a change of variable we deduce that 
		$$
		\int_{\R^2}   v(x+y,y) dx=\int_{\R^2} v(x',y) dx',
		$$
		hence the map 
		$$
		y \to  \int_{\R^2}   v(x',y) dx' 
		$$
			is bounded and uniformly continuous. 
\end{proof}
\subsection{The semigroup  of the heat equation in $\BUC(\R^2,Z)$}

Let $(Z, \Vert . \Vert _Z ) $ be a Banach space.  Let us consider the semigroup generated by the heat equation on $Y=L^1(\R^2,Z)$ given by
$$
T(t)\varphi(x)=\int_{\R^2} K(t,x-x')\varphi(x')dx',\;\forall \varphi\in \BUC(\R^2,Z).
$$
Here recall that the function $K$ is the two-dimensional heat kernel.
Now we will prove the following lemma. We refer to \cite{Haase} for more results on vectored valued elliptic operators and related evolution problems. 
\begin{lemma}\label{LE6.2}
	The semigroup $\{T(t)\}$ is strongly continuous in $Y=\BUC \left(\R^2, Z \right)$.
\end{lemma}
\begin{proof}
	First observe that
	$$
	\|T(t)\varphi\|_Y\leq \|\varphi\|_Y,\;\forall t>0,\;\forall \varphi\in Y.
	$$
	To prove the strong continuity of $T(t)$ we fix $\varphi\in Y$. We have 
	$$
	T(t)\varphi(x)-\varphi(x)=\int_{\R^2} K(t,x-x')\varphi(x')dx'-\varphi(x)=\int_{\R^2} K(t,z) \left[ \varphi(x-z)-\varphi(x) \right] dz. 
	$$
Let $\varepsilon>0$ be given. By using the fact that $x \to \varphi(x)$ is uniformly continuous, we deduce that there exists $\eta>0$ such that 
$$
\Vert \varphi(x-z)-\varphi(x) \Vert \leq \varepsilon/2, \forall x,z \in \R^2, \text{ whenever } \vert z \vert \leq \eta. 
$$
and we have 
		\begin{equation*}
		\begin{split}
			\|T(t)(\phi)(x)-\phi(x)\|_Z\leq &\int_{|z|\leq\eta} K(t,z)\left\|\phi(x-z)-\phi(x)\right\|_{Z}dz\\&+\int_{|z|\geq\eta} K(t,z)\left\|\phi(x-z)-\phi(x)\right\|_{Z}dz,
		\end{split}
	\end{equation*}
hence 
$$
	\Vert T(t)\phi(x)-\phi(x)\Vert_Z \leq \varepsilon/2  +2 \Vert \varphi \Vert_Z  \int_{|z|\geq\eta} K(t,z) dz,
$$
and since 
	$$
 \lim_{ t \to 0}\int_{|z|\geq\eta} K(t,z)=0, 
$$
the result follows. 
\end{proof}

\subsection{The semigroup  of the heat equation in $L^1(\R^2,Z)$ }

Let $(Z, \Vert . \Vert _Z ) $ be a Banach space.  We consider the heat semigroup on $Y=L^1(\R^2,Z)$ given by
$$
T(t)\varphi(x)=\int_{\R^2} K(t,x-x')\varphi(x')dx',\;\forall \varphi\in L^1(\R^2,Z).
$$
We have the following lemma. 
\begin{lemma} \label{LE6.3}
	The semigroup $\{T(t)\}$ is strongly continuous in $Y=L^1(\R^2,Z)$.
\end{lemma}

\begin{proof}
	First observe that
	$$
	\|T(t)\varphi\|_Y\leq \|\varphi\|_Y,\;\forall t>0,\;\forall \varphi\in Y.
	$$
	To prove the strong continuity of $T(t)$ we fix $\varphi\in Y$.
	Next fix $\epsilon>0$ and note that there exists $\phi\in C_c(\R^2; Z)$ (compactly supported) such that
	$$
	\int_{\R^2} \|\varphi(x)-\phi(x)\|_Z dx\leq \varepsilon/4.
	$$
	Next one has for all $t>0$
	$$
	T(t)\varphi-\varphi=T(t)(\varphi-\phi)-(\varphi-\phi)+\left[T(t)\phi-\phi\right],
	$$
	and
	$$
	\|T(t)\varphi-\varphi\|_Y  \leq \varepsilon/2+\left\|T(t)\phi-\phi\right\|_Y,
	$$
	and the result follows by using similar argument than in the proof of Lemma \ref{LE6.2}. 
\end{proof}

Now assume that $v \in Y_2$. Then $v$ can be written as 
$$
v(x,y)=  w(x-y,y)
$$  
for some  
$$
x\to  w(x,.) \in L^1 \left(\R^2, \BUC \left(\R^2\right) \right).
$$
We obtain that 
$$
\begin{array}{ll}
T(t)v(x,y)&=\int_{\R^2} K(t,x-x')v(x',y)dx', \vspace{0.2cm}\\
&=\int_{\R^2} K(t,x-x')w(x'-y,y)dx', \vspace{0.2cm}  \\
&=\int_{\R^2} K(t,x-z-y))w(z,y)dz,
\end{array}
$$
therefore 
$$
T(t)v(x+y,y)=\int_{\R^2} K(t,x-z))w(z,y)dz,
$$

$$
\Vert T(t)v \Vert_{Y_2}  = \int_{\R^2} \sup_{y \in \R}  \vert T(t)v(x+y,y) \vert dx = \Vert T(t)w  \Vert_{L^1(\R^2, \BUC \left( \R^2 \right))}  . 
$$
So by using the same argument than in the proof of Lemma \ref{LE6.3} combine with the above observations, we deduce the following lemma. 
\begin{lemma} \label{LE6.4} The semigroup $\{T(t)\}$ is strongly continuous in $Y_2$.
\end{lemma}

\subsection{Existence of a linear semigroup}
Let us consider a non-autonomous perturbation of the model return to home model 
\begin{equation} \label{6.1}
	\left\{
	\begin{array}{lll}
		\partial_t u(t,y)&= &-\gamma u(t,y),\vspace{0.2cm}\\
			\partial_t v(t,x,y)&=  \varepsilon^2 \bigtriangleup_x v(t,x,y)&-\alpha v(t,x,y),
	\end{array}
	 \right. 
\end{equation}
with the initial distribution 
\begin{equation} \label{6.3}
	\left\{ 
	\begin{array}{ll}
		u(0,y)=u_0(y) \in \BUC \left( \R^2\right),\vspace{0.1cm} \\
		\text{and} \vspace{0.1cm}\\
		 v(0,x,y)=v_0(x,y) \in Y_1 \cap Y_2.
	\end{array}
\right.
\end{equation}
We define the state space
$$
X=\BUC\left(\R^2\right) \times  Y_1 \cap Y_2,
$$
which becomes a Banach space when it is endowed with the standard product norm $\|\cdot\|_X$ given by
$$
\Vert (u,v)\Vert_X = \Vert u \Vert_{\infty}+ \Vert v \Vert_{Y_1}+\Vert v \Vert_{Y_2}. 
$$
Then we the following lemma holds.
\begin{lemma} The family of linear operator defined for $t\geq 0$ by 
\begin{equation}\label{6.5}
T_A(t) \left( 
\begin{array}{c}
	u\\
	v
\end{array}
\right)=
\left( 
\begin{array}{c}
	e^{-\gamma t}u\\
	e^{-\alpha t} \int_{\mathbb R^2} K(t,x-z)  v(z,y)d z
\end{array}
\right)
\end{equation}
defines a strongly continuous semigroup of bounded linear operators on $(X, \Vert .\Vert_X)$. 
\end{lemma}

\subsection{Bounded linear perturbation}
We define the linear operator $B:\BUC\left(\R^2\right) \times  Y_1 \cap Y_2 \to \BUC\left(\R^2\right) \times  Y_1 \cap Y_2$ by 
	$$
	B \left( 
	\begin{array}{c}
		u\\
		v
	\end{array}
	\right)=\left( 
	\begin{array}{c}
		\phi\\
		\psi
	\end{array}
	\right)\text{ with }
	\left\{
	\begin{array}{c}
	 \phi(y)=\alpha \int_{\R^2}v(x,y)dx \\
\psi(x,y)=\gamma \, \rho(x-y) \, u(y)
	\end{array}
	\right..
	$$
Note that $\phi\in 	\BUC\left(\R^2\right)$ due to Lemma \ref{LE6.1}.
	Next the following lemma holds true.
	\begin{lemma} The linear operator  $B$  is bounded on $(X,\|\cdot\|_X)$. 
		\end{lemma} 
	\begin{proof} Let $u \in \BUC(\R^2)$ and $v \in  Y_1 \cap Y_2 $.  As noticed above, from Lemma \ref{LE6.1} and since $v \in Y_2$, we know that the map $y \to	 \alpha \int_{\R^2}v(x,y)dx$ is bounded and uniformly continuous. Moreover, we have 
		$$
	\sup_{y \in \R^2} \vert \int_{\R^2}v(x,y)dx \vert=\sup_{y \in \R^2} \vert \int_{\R^2}v(x'+y,y)dx'  \vert  \leq   \int_{\R^2}\sup_{y \in \R^2} \vert v(x'+y,y) \vert dx'=\Vert v \Vert_{2},
		$$
		while
		$$
		\begin{array}{ll}
			\Vert  \psi \Vert_{Y_1}&=	\sup_{x \in \R^2} \vert \int_{\R^2} \rho(x-y) \, u(y)dy \vert \\
			&=\sup_{x \in \R^2} \vert \int_{\R^2} \rho(z) \, u(x-z)dz  \vert \\
			& \leq  \int_{\R^2} \rho(z) dz\,	\sup_{x \in \R^2} \vert u(x) \vert ,
		\end{array}
		$$
		and 
		$$
		\Vert \rho(x-y) \, u(y) \Vert_{Y_2} = \int_{\R^2} \sup_{y \in \R}  \vert \rho(x) \, u(y)  \vert dx=  \sup_{y \in \R}  \vert \, u(y)  \vert.
		$$
		The result follows. 
	\end{proof}

Let us consider the return to home model 
\begin{equation} \label{6.6}
	\left\{
	\begin{array}{l}
		\partial_t u(t,y)=-\gamma u(t,y)+ \alpha \int_{\R^2}v(t,x,y)dx, \vspace{0.2cm}\\
		\partial_t v(t,x,y)=  \varepsilon^2 \bigtriangleup_x v(t,x,y)-\alpha \,  v(t,x,y)+\gamma \, \rho(x-y) \, u(t,y),
	\end{array}
	\right. 
\end{equation}
with the initial distribution 
\begin{equation} \label{6.7}
	\left\{ 
	\begin{array}{ll}
		u(0,y)=u_0(y) \in \BUC \left( \R^2\right),\vspace{0.1cm} \\
		\text{and} \vspace{0.1cm}\\
		v(0,x,y)=v_0(x,y) \in Y_1 \cap Y_2.
	\end{array}
	\right.
\end{equation}
Let $A:D(A) \subset X \to X $ be the infinitesimal generator of the strongly continuous semigroup $\left\{T_A(t)\right\}_{t \geq 0}$.  By considering  $w=\left( 
\begin{array}{c}
	u\\
	v
\end{array}
\right)$  the problem \eqref{6.6} can be rewritten as the following abstract Cauchy problem  
\begin{equation} \label{6.8}
	w'(t)=\left(A+B\right)w(t), \text{ for } t \geqq 0, \text{ and } w(0)=\left( 
\begin{array}{c}
	u_0\\
	v_0
\end{array}
\right) \in X. 	
\end{equation}

	\begin{theorem} \label{TH6.7}The linear operator $A+B:D(A) \subset X \to X $ is the infinitesimal generator of a strongly continuous semigroup $\left\{T_{A+B}(t)\right\}_{t \geq 0}$ of bounded linear operator on $X$. Moreover $t \to u(t)=	T_{A+B}(t)w_0$ is the unique continuous map satisfying  
		$$
	u(t) =	T_{A}(t)w_0 + \int_{0}^{t} T_{A}\left(t-\sigma \right) Bu(\sigma )d \sigma, \forall t \geq 0. 
		$$
		Furthermore, for each $t \geq 0$, the positive cone $X_+=\BUC_+\left(\R^2\right) \times  Y_{1+} \cap Y_{2+}$ (where the plus sign stands for the space of non-negative functions) is positive invariant by $T_{A+B}(t)$. 
	\end{theorem}

\section{Model with colonization and return  to home}

In this section we incorporate a colonization phenomenon in the return to home model. This new feature turns out to be important especially when dealing with animal dispersal behavior. The following model takes into account a colonization effect that corresponds to the fact that when it is time to return home for the travelers, a fraction $p$ (with $0 \leq p \leq 1$) returns back home like before, and a fraction $1-p$ colonizes the location they have reached at the end of their traveling period. This return to home model with colonization reads as the following system of equations
\begin{equation} \label{7.1}
	\left\{
	\begin{array}{ll}
		\partial_t u(t,y)&=-\gamma u(t,y)+\, \alpha \, p \int_{\R^2}v(t,x,y)dx\vspace{0.2cm}\\
		&+ \underset{\text{Flux of individuals colonizing the location y}}{\underbrace{ \alpha \, (1-p) \int_{\R^2}v(t,y,y')dy'}} \vspace{0.4cm}\\
		\partial_t v(t,x,y)&=  \varepsilon^2 \bigtriangleup_x v(t,x,y) -\alpha \,  v(t,x,y)+\gamma \, \rho(x-y) \, u(t,y),
	\end{array}
	\right. 
\end{equation}
where $p \in [0,1]$. 

\begin{remark} Recall that $1/\gamma$ is the average time spent at home while $1/\alpha$ stands for the average time spent to travel. If we assume that $p=0$ (that is no individual return to home after traveling), the process described above corresponds to rest during a period $1/\gamma$ after a period of travel $1/\alpha$. So the case where $p=0$ is also realistic and this extends the standard diffusion process. The corresponding system could called for short the "diffuse and rest model". 
	\end{remark} 

In the model the term $ \alpha \, (1-p) \int_{\R^2}v(t,y,y')dy'$ counts the fraction of individuals who colonize the location $y$ at the end of their travel. It is summed over all their previous home locations $y'\in\R^2$.  Hence $1-p$ is the fraction of the travelers that will change their home location by the end of their travel period. 

\begin{lemma} \label{LE7.1} Let $v \in Y_1$ be given. Then the map  
	$$
	x \to \int_{\R^2}    v(x,y) dy
	$$
	is bounded and uniformly continuous on $\R^2$, and one has
	$$
	\sup_{x \in  \R^2}  \left\vert \int_{\R^2}    v(x,y) dy \right\vert \leq \Vert v \Vert_{Y_1}. 
	$$
	In other words the linear map $v\to \displaystyle \int_{\R^2}v(\cdot,y)dy$ is bounded from $Y_1$ into $\BUC  \left(\R^2\right)$.
\end{lemma} 
\begin{proof} Recall that the space $Y_1$ is the set of maps $x \to v(x,.)$ which belongs to $\BUC  \left(\R^2, L^1 \left(\R^2\right) \right)$. The linear form  
	$$
	G:	u \to G(u)=\int_{\R^2}    u(y) dy .
	$$
	is bounded linear on $L^1(\R^2)$. Therefore the map $x \to G(u(x,.))$ must be bounded and uniformly continuous on $\R^2$. 
\end{proof}

\bigskip 
\noindent We define the linear bounded operator $C:X \to X$ by 
$$
C \left( 
\begin{array}{c}
	u\\
	v
\end{array}
\right)=\left( 
\begin{array}{c}
	\phi\\
	0
\end{array}
\right)
\text{ with }\phi(y)=\alpha \, \int_{\R^2}v(y,y')dy'.
$$
By considering  $w=\left( 
\begin{array}{c}
	u\\
	v
\end{array}
\right)$ the problem \eqref{7.1} can be rewritten as the following abstract Cauchy problem in $X$
\begin{equation*}
	w'(t)=\left(A+pB+(1-p)C\right)w(t), \text{ for } t \geqq 0, \text{ and } w(0)=w_0=\left( 
	\begin{array}{c}
		u_0\\
		v_0
	\end{array}
	\right) \in  X. 
\end{equation*}

\begin{theorem} \label{TH7.2}Let $p\in [0,1]$ be given. Then the linear operator $A+pB+(1-p)C:D(A) \subset X \to X $ is the infinitesimal generator of a strongly continuous semigroup $\left\{T_{A+pB+(1-p)C}(t)\right\}_{t \geq 0}$ of bounded linear operator on $X$. Moreover $t \to w(t)=	T_{A+pB+(1-p)C}(t)w_0$ is the unique continuous map satisfying  
	$$
	w(t) =	T_{A}(t)w_0 + \int_{0}^{t} T_{A}\left(t-\sigma \right) \left( pB+(1-p)C \right) w(\sigma )d \sigma,\; \forall t \geq 0. 
	$$
	Furthermore, for each $t \geq 0$, the positive cone $X_+=\BUC_+\left(\R^2\right) \times  \left(Y_{1+} \cap Y_{2+}\right)$ (where the plus sign stands for the space of non-negative functions) is positively invariant by $T_{A+pB+(1-p)C}(t)$. 
\end{theorem}

\medskip 
\noindent \textbf{Colonization and return to home model:} If add a vital dynamic on the individual staying at home in the model with colonization and retrun home 
\begin{equation} \label{7.2}
	\left\{
	\begin{array}{lll}
		\partial_t u(t,y)=&-\gamma u(t,y)+\, \alpha  \left[p \int_{\R^2}v(t,x,y)dx +  (1-p) \int_{\R^2}v(t,y,y')dy'\right]   \vspace{0.4cm}\\
		&+ \beta u(t,y)- \mu u(t,y)- \kappa u(t,y)^2 \vspace{0.4cm}\\
		\partial_t v(t,x,y)=&  \varepsilon^2 \bigtriangleup_x v(t,x,y) -\alpha \,  v(t,x,y)+\gamma \, \rho(x-y) \, u(t,y)-\nu v(t,x,y),
	\end{array}
	\right. 
\end{equation}
where $p \in [0,1]$. 

To deal with the above system we consider the nonlinear map $F:X\to X$ given by
$$
F\begin{pmatrix} u\\ v\end{pmatrix} =\begin{pmatrix}(\beta-\mu)u-\kappa u^2\\ 0\end{pmatrix}.
$$
Together with the notation of the previous section and replacing $\alpha$ by $\alpha+\nu$ into \eqref{6.5}, problem \eqref{7.2} rewrites as the following abstract Cauchy problem for $w=\left( 
\begin{array}{c}
	u\\
	v
\end{array}
\right)\in X$:
\begin{equation}\label{7.3}
w'(t)=\left(A+pB+(1-p)C\right)w(t)+F(w(t)), \text{ for } t \geqq 0,
\end{equation}
and
\begin{equation}\label{7.4}
 w(0)=w_0=\left( 
		\begin{array}{c}
			u_0\\
			v_0
		\end{array}
		\right) \in  X. 
\end{equation}

Note that the function $f:\R\to\R$ given by $f(u)=(\beta-\mu)u-\kappa u^2$ satisfies: For each $M>0$ there exists $\lambda(M)\in \R$ such that the function $u\to f(u)+\lambda(M)u$ is non-negative and increasing on $[0,M]$. 
Hence since $F$ is locally Lipschitz continuous on $X$ and the semigroup $\{T_{A+pB+(1-p)C}(t)\}$ is positive (with respect to the cone $X_+$) we obtain the following lemma.

The following theorem is a consequence of the results proved in Magal and Ruan \cite{MR09, MR18}. The result on monotone semiflows are consequences of Hal Smith \cite{S1, S2}. We also refer to Magal, Seydi, and Wang \cite{MSW} for recent extensions. 
\begin{theorem}
The Cauchy problem \eqref{7.3}-\eqref{7.4} generates a maximal semiflow $U: \R_+\times X_+ \to X_+$ on $X_+$. More precisely 
for each $w_0 \in \X_+$ 
Moreover $U$ is monotone on $X_+$, that is  
$$
w_1 \geq w_0 \geq 0 \Rightarrow  U(t)w_1 \geq U(t)w_0 \geq 0, \forall t \geq 0.  
$$
\end{theorem}

\begin{remark} In the Lagrangian system of coordinate the above system becomes  
\begin{equation} \label{7.4}
	\left\{
	\begin{array}{lll}
		\partial_t u(t,y)=&-\gamma u(t,y)+\, \alpha  \left[p \int_{\R^2}w(t,z,y)dx +  (1-p) \int_{\R^2}w(t,y+y',y')dy'\right]   \vspace{0.4cm}\\
		&+ \beta u(t,y)- \mu u(t,y)- \kappa u(t,y)^2 \vspace{0.4cm}\\
		\partial_t w(t,z,y)=&  \varepsilon^2 \bigtriangleup_z w(t,z,y) -\alpha \,  w(t,z,y)+\gamma \, \rho(z) \, u(t,y)-\nu w(t,z,y). 
	\end{array}
	\right. 
\end{equation}
\end{remark}

\section{Application the an epidemic model with return to home}
\noindent \textbf{Epidemic model for people staying at home:} The following equation describes the flux between individuals staying at home at the location $y \in \R^2$ and individuals out of the house 
\begin{equation} \label{8.1}
	\left\{ 
	\begin{array}{ll}
		\partial_t  s_1(t,y)=& \alpha \int_{\R^2} s_2(t,x,y)dx- \gamma s_1(t,y)-\kappa_1 s_1(t,y) i_1(t,y), \\
		\\
		\partial_t i_1(t,y)= &\alpha \int_{\R^2} i_2(t,x,y)dx- \gamma i_1(t,y)+\kappa_1 s_1(t,y) i_1(t,y) \vspace{0.2cm} \\
		&-\nu_1 i_1(t,y)
	\end{array}
\right.
\end{equation}
\medskip 
\noindent \textbf{Epidemic model for travelers:} The following equation describes the flux between individuals staying at home at the location $y$ and the travelers 
\begin{equation} \label{8.2}
\left\{ 
	\begin{array}{ll}
		\partial_t s_2(t,x,y)&=  \varepsilon^2 \bigtriangleup_x s_2(t,x,y) -  \alpha \, s_2(t,x,y)+ \gamma \, \rho(x-y) \, s_1(t,y) 
		\\	
		\\&-\kappa_2 s_2(t,x,y) \int_{ \R^2} i_2(t,x, \widehat{y}) d\widehat{y} \\
		\\
		\partial_t i_2(t,x,y)&=  \varepsilon^2 \bigtriangleup_x i_2(t,x,y) -  \alpha \, i_2(t,x,y)+ \gamma \, \rho(x-y) \, i_1(t,y) 
		\\
		\\&+ \kappa_2 s_2(t,x,y) \int_{ \R^2}  i_2(t,x, \widehat{y}) d\widehat{y}-\nu_2 i_2(t,x,y)
	\end{array}
\right.
\end{equation}
In the above epidemic model we assume that the transmission occurs locally in space. At home (see \eqref{8.1}) the term   $\kappa_1 s_1(t,y) i_1(t,y)$  means that susceptible individuals located at $y$ can only be infected by the infectious at home and located at the same position $y$. For travelers susceptible individuals originated from $y$ and located at $x$ at time $t$ can be infected by infected travelers located at $x$ at time $t$ whatever the location of their home position. In words, for the travelers the incidence rate at time $t$ and spatial location $x$ reads as $\kappa_2 s_2(t,x,y) \int_{\R^2} i_2(t,x, \widehat{y}) d\widehat{y}$.

\medskip 
\noindent \textbf{Initial condition:} The system  \eqref{8.1} and \eqref{8.2} is complemented with the initial distributions
\begin{equation} \label{8.3}
 s_1(0,y):=s_{10}(y) \text{ and } i_1(0,y):=i_{10}(y), 
\end{equation}
and
\begin{equation} \label{8.4}
	s_2(0,x,y):=s_{20}(x,y) \text{ and } i_2(0,x,y):=i_{20}(x,y). 
\end{equation}
As above to handle this problem we consider the Banach space $\mathbb X:=X\times X$ endowed with the product norm. We define the linear operator $\mathbb A: D(\mathbb A)\subset \mathbb X\to \mathbb X$ by
\begin{equation*}
D(\mathbb A)=D(A)\times D(A),
\end{equation*}
and
\begin{equation*}
\mathbb A\begin{pmatrix} \begin{pmatrix} s_1\\ s_2\end{pmatrix}\\\begin{pmatrix} i_1\\ i_2\end{pmatrix}\end{pmatrix}= \begin{pmatrix} 
\begin{pmatrix}
- \gamma s_1(y)\\
\varepsilon^2 \bigtriangleup_x s_2(x,y) -  \alpha \, s_2(x,y)+ \gamma \, \rho(x-y) \, s_1(y)
\end{pmatrix}
\\
\begin{pmatrix}
- (\gamma+\nu_1) i_1(y)\\
\varepsilon^2 \bigtriangleup_x i_2(x,y) -  (\alpha+\nu_2)  i_2(x,y)+ \gamma \, \rho(x-y) \, i_1(y)
\end{pmatrix}
\end{pmatrix}
\end{equation*}
We also define the bounded linear operator $\mathbb L:X\times X\to X\times X$ by
\begin{equation*}
\mathbb L\begin{pmatrix} \begin{pmatrix} s_1\\ s_2\end{pmatrix}\\\begin{pmatrix} i_1\\ i_2\end{pmatrix}\end{pmatrix}= \begin{pmatrix} 
\begin{pmatrix}
\alpha \int_{\R^2} s_2(x,\cdot)dx\\
0
\end{pmatrix}
\\
\begin{pmatrix}
\alpha \int_{\R^2} i_2(x,\cdot)dx
\\ 0
\end{pmatrix}
\end{pmatrix}
\end{equation*}
so that $\mathbb A+\mathbb L:D(\mathbb A):\mathbb X\to \mathbb X$ is the infinitesimal generator of a strongly continuous semigroup on $\mathbb X$, leaving the positive cone $\mathbb X_+=X_+\times X_+$ positively invariant. 

For the contamination terms, we define the nonlinear map $\mathbb F:X\times X\to X\times X$ by
\begin{equation*}
\mathbb F\begin{pmatrix} \begin{pmatrix} s_1\\ s_2\end{pmatrix}\\\begin{pmatrix} i_1\\ i_2\end{pmatrix}\end{pmatrix}= \begin{pmatrix} 
\begin{pmatrix}
-\kappa_1 s_1(y) i_1(y)\\
-\kappa_2 s_2(x,y) \int_{ \R^2} i_2(x, \widehat{y}) d\widehat{y}
\end{pmatrix}
\\
\begin{pmatrix}
\kappa_1 s_1(y) i_1(y)\\
\kappa_2 s_2(x,y) \int_{ \R^2} i_2(x, \widehat{y}) d\widehat{y}
\end{pmatrix}
\end{pmatrix}
\end{equation*}

Next setting
$$
W=\begin{pmatrix} \begin{pmatrix} s_1\\ s_2\end{pmatrix}\\\begin{pmatrix} i_1\\ i_2\end{pmatrix}\end{pmatrix}\in \mathbb X,
$$
system \eqref{8.1}-\eqref{8.4} becomes
\begin{equation}\label{8.5}
W'(t)=\left(\mathbb A+\mathbb L\right)W(t)+\mathbb F\left(W(t)\right),\;t>0,
\end{equation}
together with
\begin{equation} \label{8.6}
W(0)=\begin{pmatrix} \begin{pmatrix} s_{10}\\ s_{20}\end{pmatrix}\\\begin{pmatrix} i_{10}\\ i_{20}\end{pmatrix}\end{pmatrix}\in \mathbb X.
\end{equation}

Since $\mathbb F$ is locally Lipschitz continuous on $\mathbb X$, \eqref{8.5} generates a strongly continuous maximal semiflow on $\mathbb X$. Moreover for each $M>0$ there exists $\lambda(M)$ such that for all $W\in \mathbb X_+$ with $\|W\|_{\mathbb X}\leq M$ one has
$$
\mathbb F\left(W\right)+\lambda(M) W\in \mathbb X_+.
$$  
As a consequence, since $\mathbb A+\mathbb L$ generates a strongly continuous semigroup on $\mathbb X$ which is positive (with respect to the positive cone $\mathbb X_+$), the cone $\mathbb X_+$ is positively invariant with respect to this maximal semiflow. In other words, when the initial data are positive then the maximal solution is positive as well.

To see that the semiflow is globally defined, fix and initial data $
W(0)=( (s_{10}, s_{20}), (i_{10}, i_{20}))^T\in \mathbb X_+$ and denote by $W=( (s_{1}, s_{1}), (i_{1}, i_{2}))^T\in \mathbb X_+$ the maximal solution with initial data $w(0)$ defined on $[0,\tau)$ with $\tau\in (0,\infty]$. Let us that $\tau=\infty$
Adding the equations \eqref{8.1} and \eqref{8.2} integrated with respect to $x\in\R^2$ yields for any $t\geq 0$ and $y\in\R^2$:
\begin{equation*}
\begin{split}
s_1(t,y)+i_1(t,y)&\leq s_1(t,y)+i_1(t,y)+\int_{\R^2}\left(s_2+i_2\right)(t,x,y)dx\\
&\leq M:=\left\|(s_{10},s_{20})^T\right\|_X+\|(i_{10},i_{2,0})^T\|_X.
\end{split}
\end{equation*} 
Now from \eqref{8.2} and the positivity of the solution we obtain for $t\in [0,\tau)$, $(x,y)\in\R^2\times\R^2$ that
\begin{equation*}
\begin{split}
0\leq s_2(t,x,y)&\leq \int_{\R^2} e^{-\alpha t}K(t,x-x')s_{20}(x',y)dx'\\
&+\gamma M\int_0^t e^{-\alpha s}\int_{\R^2}K(s,x-y-x')\rho(x')dx,'
\end{split}
\end{equation*} 
This proves that the map $t\to \|s_2(t,\cdot)\|_{Y_1}+\|s_2(t,\cdot)\|_{Y_2}$ is uniformly bounded on $[0,\tau)$. 
From these bounds, the equation for $i_2$ becomes sub-linear so that $t\to \|i_2(t,\cdot)\|_{Y_1\cap Y_2}$ has at most an exponential growth. This prevents from finite time blow-up and ensures that $\tau=\infty$ and the semiflow is globally defined on $\mathbb X_+$. 

The following theorem is a consequence of the results proved in Magal and Ruan \cite{MR09, MR18}. 
\begin{theorem} The Cauchy problem \eqref{8.5}-\eqref{8.6} generates a continuous semiflow $U: \R_+ \times \mathbb X_+ \to \mathbb X_+$. More precisely, for each $x \in \X_+$, the map $t \to U(t)x$ is the unique continuous map from $[0, \infty)$ into $\X$ satisfying  
	$$
	u(t)=T_{{\A+\LL}}(t)x + \int_0^t T_{{\A+\LL}}(t-s)  \F(u(s)) ds, \forall t \geq 0. 
	$$
\end{theorem}


\begin{thebibliography}{99}

%\bibitem{coville} J. Coville, L. Dupaigne, On a non-local equation arising in population dynamics, Proc. Roy. Soc. Edinburgh Sect. A 137 (2007), 727-755.


	\bibitem{BHG} 	Brockmann, D., Hufnagel, L., \& Geisel, T. (2006). The scaling laws of human travel. Nature, 439(7075), 462-465.
	
\bibitem{Cantrell} Cantrell, R. S., and Cosner, C. (2004). Spatial ecology via reaction-diffusion equations. John Wiley \& Sons. 

\bibitem{Cantrell-Cosner-Ruan} Cantrell, R. S., Cosner, C., \& Ruan, S. (Eds.). (2010). Spatial ecology. CRC Press. 
	
	
		\bibitem{CBCI} Cosner, C., Beier, J. C., Cantrell, R. S., Impoinvil, D., Kapitanski, L., Potts, M. D., ... \& Ruan, S. (2009). The effects of human movement on the persistence of vector-borne diseases. Journal of Theoretical Biology, 258(4), 550-560. 
		
	\bibitem{F}  Fisher, R.A. (1937). The wave of advance of advantageous genes. Annals of Eugenics 7(4), 355–369.
	
\bibitem{GHB} 	 Gonzalez, M. C., Hidalgo, C. A., \& Barabasi, A. L. (2008). Understanding individual human mobility patterns. Nature, 453(7196), 779-782.

\bibitem{KSZ}  Klafter, J., Shlesinger, M. F., \& Zumofen, G. (1996). Beyond brownian motion. Physics today, 49(2), 33-39.

\bibitem{KPP}	Kolmogorov, A.N., Petrovski, I.G., Piskunov, N.S. 	(1937). Étude de l’équation de	la diffusion avec croissance de la quantité de matière et son application à un problème biologique. Bull. Univ. Moskow, Ser. Internat., Sec. A 1, 1–25.



\bibitem{Murray} Murray, J. D. (2001). Mathematical biology II: spatial models and biomedical applications (Vol. 3). New York: Springer.

	
	%	\bibitem{Gao} 	Gao, D. How does dispersal affect the infection size?.\textit{ SIAM Journal on Applied Mathematics}, 80(5) (2020), 2144-2169.
	
	\bibitem{Haase} M. Haase, The functional calculus for sectorial operators, Birkhäuser Basel (2006).
%	
%\bibitem{Lutscher} F. Lutscher, E. Pachepsky, M. Lewis, The effect of dispersal patterns on stream populations. SIAM J. Appl. Math. 65 (2005), 1305-1327.

	\bibitem{MR09} Magal, P., \& Ruan, S. (2009). On semilinear Cauchy problems with non-dense domain. Advances in Differential Equations, 14(11/12), 1041-1084.


	\bibitem{MR18} Magal, P., \& Ruan, S. (2018). Theory and applications of abstract semilinear Cauchy problems. Springer International Publishing.
	
	\bibitem{MSW} 	Magal, P., Seydi, O., \& Wang, F. B. (2019). Monotone abstract non-densely defined Cauchy problems applied to age structured population dynamic models. Journal of Mathematical Analysis and Applications, 479(1), 450-481.
	
\bibitem{MWW1} 	P. Magal, G. F. Webb and Y. Wu (2019) An Environmental Model of Honey Bee Colony Collapse Due to Pesticide Contamination, Bulletin of Mathematical Biology , 81, 4908–4931. 	
	
\bibitem{MWW2} 	Magal, P., Webb, G. F., \& Wu, Y. (2020). A spatial model of honey bee colony collapse due to pesticide contamination of foraging bees. Journal of mathematical biology, 80, 2363-2393.


	
	\bibitem{MS}   Mantegna, R. N., \& Stanley, H. E. (1994). Stochastic process with ultraslow convergence to a Gaussian: the truncated Lévy flight. Physical Review Letters, 73(22), 2946.
	
	\bibitem{Perthame} Perthame, B. (2015). Parabolic equations in biology. In Parabolic Equations in Biology (pp. 1-21). Springer, Cham.
	
\bibitem{Roques} Roques, L. (2013). Modèles de réaction-diffusion pour l'écologie spatiale: Avec exercices dirigés. Editions Quae. 

	\bibitem{Ruan} S. Ruan (2007), Spatial-Temporal Dynamics in Nonlocal Epidemiological Models, in ``Mathematics for Life Science and Medicine'', Y. Takeuchi, K. Sato and Y. Iwasa (eds.), Springer-Verlag, Berlin, pp. 97-122 . 
	
	\bibitem{Ruan2017} Ruan, S. (2017). Spatiotemporal epidemic models for rabies among animals. Infectious disease modelling, 2(3), 277-287. 

	\bibitem{S1}  	Smith, H. L. (2008). Monotone dynamical systems: an introduction to the theory of competitive and cooperative systems: an introduction to the theory of competitive and cooperative systems (No. 41). American Mathematical Soc..
	
	\bibitem{S2}   Smith, H. L. (2017). Monotone dynamical systems: reflections on new advances \& applications. Discrete \& Continuous Dynamical Systems-A, 37(1), 485.
	
	
		\bibitem{Data} United States Census Bureau:  \url{https://www.census.gov/data/datasets/time-series/demo/popest/2010s-counties-total.html#par_textimage_70769902}
	


	
\end{thebibliography}
\end{document}